\documentclass[a4paper, 10pt]{amsart}
\usepackage{amssymb, amsmath}
\usepackage{amscd}
\usepackage[alwaysadjust]{paralist}
\usepackage[capitalize]{cleveref}
\usepackage{enumitem}
\usepackage{esint}
\usepackage{graphicx} 
\usepackage[rgb]{xcolor} 
\usepackage{mathtools}

\usepackage{mathtools}

\numberwithin{equation}{section}

\newtheorem{cor}[equation]{Corollary}
\newtheorem{lemma}[equation]{Lemma}
\newtheorem{remark}[equation]{Remark}

\newtheorem{theorem}[equation]{Theorem}
\theoremstyle{definition}

\newtheorem{exa}[equation]{Example}

\newtheorem*{dir2}{Dir(n,m,$\rho$,$f_k$)}

\newtheorem*{acknowledge}{Acknowledgments}
\newtheorem*{thma}{Theorem A}
\newtheorem*{thmb}{Theorem B}

\newcommand{\R}{\mathbb{R}}
\newcommand{\N}{\mathbb{N}}
\newcommand{\Z}{\mathbb{Z}}
\newcommand{\Sph}{\mathbb{S}}
\begin{document}

\title[Harmonic maps]{A Dirichlet problem on balls}

\author{Anna Siffert}
\address
{Max Planck Institute for Mathematics,
Vivatsgasse 7, 53111 Bonn.}
\email{siffert@mpim-bonn.mpg.de}
\date{\today}

\begin{abstract}
By generalizing a result of J\"ager and Kaul we provide infinitely many solutions of a Dirichlet problem on balls.
\end{abstract}

\maketitle

\section*{Introduction}
The purpose of this paper is the construction of harmonic maps from balls to spheres which are solutions of a Dirichlet problem.
\smallskip

Recall that the energy of a smooth map $\varphi:N\rightarrow M$ between two Riemannian manifolds, $(M,g)$ and $(N,h),$ 
 is defined to be
$$E(\varphi)=\tfrac{1}{2}\int_N\lvert d\varphi\rvert^2\omega_g,$$
where $\omega_g$ denotes the volume density on $N$ and $\lvert\,\cdot\,\lvert$ denotes the Hilbert Schmidt norm.
A map $\varphi$ is called harmonic if it is a critical point of the energy functional.
For our purposes we also need to consider harmonic maps in the subset $H^1_2(B^n,\Sph^n)$ of the Sobolev space $H^1_2(B^n,\R^{n+1})$.
The definition is roughly analogous to the one in the smooth case, but more technical -- details can be found in Subsection\,1.2.2 or in \cite{hamaps,jk}.

\bigskip

In this manuscript, we
construct solutions of a Dirichlet problem by generalizing the work \cite{jk} by J\"ager and Kaul.
These authors consider the following Dirichlet problem for maps from ball the compact unit ball $B^n\subset\R^n$ to the unit sphere $\Sph^n$.

\medskip
 \begin{tabbing}
\textbf{$\mbox{Dir}(n,\rho)$:} \=  For given $\rho\in[0,\pi]$, find a rotationally symmetric harmonic map\\\>  $u\in H^1_2(B^n,\Sph^n)$
which covers the north pole and has boundary values \\
  \> $b_{\rho}:\partial B^n\rightarrow\Sph^n$
given by $b_{\rho}(x)=(x\sin\rho,\cos\rho)$.
  \end{tabbing}

\medskip

\noindent
Here a map $u:B^n\rightarrow\Sph^n$ is called rotationally symmetric if it is of the form
$$u(x)=(\sin \Phi(\lvert x\lvert)\,x/\lvert x\lvert,\cos\Phi(\lvert x\lvert)),$$ 
where $\Phi:B^n\rightarrow[0,\pi]$ is assumed to be contained in $C^2$.
It is easy to show that this problem reduces to a second order ordinary differential equation for $\Phi$.
By a straightforward analysis of this differential equation, J\"ager and Kaul show that, for each $3\leq n\leq 6$, there exists an explicit $\rho_n\in[\pi/2,\pi]$, such that $\mbox{Dir}(n,\rho)$ has at least one continuous solution for all $\rho\leq\rho_n$ and no solution if $\rho>\rho_n$.
Furthermore, for $\rho=\pi/2$, the authors prove that there exist infinitely many 
solutions to the above Dirichlet problem.
In contrast to that, for $n\geq 7$, the Dirichlet problem $\mbox{Dir}(n,\rho)$ does not admit any solution for $\rho\geq\pi/2$.

\smallskip

In the present paper, we consider more general maps between balls and spheres.
Recall first that a smooth harmonic map $f:\Sph^{n-1}\rightarrow\Sph^{m-1}$
is called an eigenmap if its energy density $\lvert df\lvert^2/2$ is constant. 
(This notation can be explained as follows. For $\iota:\Sph^{m-1}\rightarrow\R^m$ the standard inclusion, the map $\phi:=\iota\circ f$ satisfies $\Delta\phi=\lvert d\phi\lvert^2\phi$, i.e. is an eigenfunction of the Laplacian.)
It is well-known, see e.g. Chapter VIII in \cite{es}, that $f$ is a harmonic eigenmap if and only if its
components are harmonic polynomials of common degree $k$. The energy density of $f$ is then given by $k(n-2+k)/2$.\\
The maps between balls and spheres we will study are of the form 
$$u(x)=(\sin \Phi(\lvert x\lvert)f_{\frak{e}_k}(x/\lvert x\lvert),\cos\Phi(\lvert x\lvert)),$$ 
where $\Phi:B^n\rightarrow[0,\pi]$ and $f_{\frak{e}_k}:\Sph^{n-1}\rightarrow\Sph^{m-1}$ is an eigenmap with energy density $\frak{e}_k=k(k+n-2)/2$, $k\in\N$.
Henceforth, we will refer to such maps as \textit{$k$-rotationally symmetric maps}.
We will deal with the following Dirichlet problem.

\medskip

\begin{tabbing}
\textbf{$\mbox{Dir}(n,m,\rho,e_k)$:}  \= For given $\rho\in[0,\pi]$, find a $k$-rotationally symmetric harmonic map\\  \>$u\in H^1_2(B^n,\Sph^m)$
which covers the north pole and has boundary \\ \> values $b_{\rho}:\partial B^n\rightarrow\Sph^m$, $b_{\rho}(x)=(f_{\frak{e}_k}(x)\sin(\rho),\cos(\rho))$.
\end{tabbing}

\medskip

\noindent Modifying methods introduced in \cite{jk}, we prove the following theorem.

\begin{thma}
\label{dirichlet}
Let $k_0:=\lfloor 2(1+k+\sqrt{k})\rfloor$.
For each $3\leq n\leq k_0$ there exists an explicit $\rho_n\in[\pi/2,\pi]$, such that $\mbox{Dir}(n,m,\rho,f_{\frak{e}_k})$ has at least one solution for $\rho\leq\rho_n$ and no solution if $\rho>\rho_n$.
Furthermore, for $\rho=\pi/2$, there exist infinitely many 
solutions to the Dirichlet problem $\mbox{Dir}(n,m,\rho,f_{\frak{e}_k})$.
However, for $n>k_0$, the Dirichlet problem $\mbox{Dir}(n,m,\rho,f_{\frak{e}_k})$ does not admit any solution for $\rho\geq\pi/2$.
\end{thma}

By relaxing the conditions on the maps $u$,
we construct further harmonic maps from balls to spheres of even dimension and obtain the following theorem.

\begin{thmb}
Let eigenmaps $f_{\frak{e}_k}:\Sph^{n-1}\rightarrow\Sph^{2m-1}$ with energy density $\frak{e}_k=k(k+n-2)/2$, $k\in\N$, be given.
 Then there exist infinitely many harmonic maps $u:B^{n}\rightarrow\Sph^{2m}$  in $H^{1}_{2}(B^{n},\R^{2m+1})$ for every $n\geq 3$.
\end{thmb}

The paper is organized as follows.
In the first section, we collect tools and results needed in later sections.
We in particular give a short introduction to harmonic maps.
The first subsection of Section\,2 contains the proof of Theorem A, while the second subsection contains the proof of Theorem B.

\begin{acknowledge}
I am very grateful to Werner Ballmann for valuable advises.
Furthermore, I would like to thank the Max Planck Institute for Mathematics in Bonn for support and excellent working conditions.
\end{acknowledge}

%%%%%%%%%%%%%%%%%%%%%%%%%%%%%%%%%%%%%%%%

\section{Preliminaries}
In this section we provide basic tools and results needed in later sections.
The first subsection contains the motivation for this paper as well as 
a very brief introduction to harmonic maps between Riemannian manifolds.
In the second subsection we consider harmonic maps from warped products into spherically symmetric manifolds, since
maps between such model spaces are precisely the objects we are studying in this paper. 
For maps between model spaces we determine the energy functional and the associated Euler Lagrange equations.
Finally, we summarize known results about the existence of harmonic maps between the above mentioned model spaces.

%%%

\subsection{Harmonic maps between Riemannian manifolds}
In this subsection we provide some background on harmonic maps between Riemannian manifolds.
We focus on smooth harmonic maps for reasons of compactness.
Subsection\,1.2.2 contains a  paragraph on harmonic maps in $H^1_2(B^n,\Sph^n)$, which establishes all 
information needed on this subject for the present paper.
For a thorough introduction to harmonic maps in Sobolev spaces we refer the reader to \cite{hamaps}.

\smallskip

The classification and construction of harmonic maps between Riemannian manifolds has been pursued by generations of mathematicians, see e.g. \cite{eells1,eells2,er,es,smith}. 
Consequently, in this subsection we focus on results which have a direct relevance for this paper;
i.e. we discuss the specific question we are interested in and the method of attack.
For a detailed introduction to harmonic maps we refer the reader to the fantastic works \cite{er}, by Eells and Ratto, or  \cite{hamaps} by H\'elein and Wood.

\smallskip

First of all we want to mention that in the generic case the construction of harmonic maps is a hard problem.
This is due to the fact that the Euler Lagrange equations of the energy functional, which are usually denoted by
\begin{align*}
\tau(\varphi)=0,
\end{align*}
where $\tau(\varphi):=\mbox{trace}\nabla d\varphi$ is the so-called tension field of $\varphi$,
constitute a (system of) elliptic, semi-linear partial differential equation of second order.
Clearly, such an equation is difficult to solve in general.

\smallskip

The question we are interested in in the present manuscript was initiated by Eells and Sampson \cite{es}, whose work
is one of cornerstones in the study of harmonic maps.
Namely, these authors started the study of the following very important question.
\begin{multline*}
\mbox{
\textit{Does every homotopy class of maps between Riemannian}}\\ \mbox{\textit{manifolds admit a harmonic representative?}}
\end{multline*}
For the special case that the target manifold is compact and all its sectional curvatures are nonnegative, Eells and Sampson answered this question in affirmative.
However, for the case that the target manifold also admits positive sectional curvatures the answer to this question is still only known in special cases. Even for maps between spheres this question is far from being completely solved, see e.g. the paper \cite{gastel} by Gastel.

\smallskip

One of the main techniques to construct solutions of the equation $\tau(\varphi)=0$ 
is to impose some additional symmetry for the manifolds $M$ and $N$, as well as for the maps between them, in order to simplify the original problem.
 All approaches relying on this idea are frequently subsumed under the notion of \textit{reduction approach} or \textit{reduction method}.
 \\
   One example for a reduction approach is the study of maps from warped products into spherically symmetric manifolds, see Subsection\,\ref{red}
for more details.
 Another example is the study of equivariant maps between cohomogeneity one manifolds, see \cite{urakawa} and \cite{ps} for more details. In both examples the condition $\tau(\varphi)=0$ reduces to a singular, second order boundary value problem.
Although there does not exist a general solution theory for such problems, 
one can show the existence of solutions in special cases, see e.g. \cite{BC,wc,siffert,su3}. 

\smallskip

For an introduction to the theory of reduction methods and further examples we refer the reader to \cite{er}.

%%%
\subsection{Harmonic maps from warped products into spherically symmetric manifolds}
\label{red}

In this subsection we collect facts about maps from a warped product into a spherically symmetric manifold.
The first part of this subsection contains the determination of the energy functional and the associated Euler Lagrange equations for maps from a warped product into a spherically symmetric manifold.
In the second part, we provide an overview over the present state of art in the study of the harmonicity of such maps.

\subsubsection{Energy functional and Euler Lagrange equations}
We determine the Euler Lagrange equations of maps from warped products into a spherically symmetric manifolds. 

\smallskip

Throughout this paper the domain manifolds of the harmonic maps will be warped products of open intervals $\frak{I}=(0,s)\subset\R$ with arbitrary Riemannian manifolds $(N^n,h)$, i.e. manifolds of the form 
$$(\frak{I}\times_f N,dr^2+f(r)^2h),$$
where $f:\frak{I}\rightarrow\R$ is some positive function.

\smallskip

The target manifolds of the harmonic maps will be spherically symmetric manifolds.
Recall that a Riemannian manifold $(M^m,g)$ is called spherically symmetric with respect to $p_0\in M$, the so-called pole, if $M^m$ has a rotation symmetry with respect to $p_0$.
In polar coordinates $(r,\theta)$ centered at $p_0$, where $\theta\in\Sph^{m-1}$, the metric of $M$ can be written as
\begin{align*}
dr^2+\varphi(r)^2d\theta^2_{\lvert\Sph^{m-1}},
\end{align*}
with $d\theta^2_{\lvert\Sph^{m-1}}$ denoting the standard metric on $\Sph^{m-1}$ and $\varphi$ is a smooth positive function satisfying $\varphi(0)=0$ and $\varphi'(0)=1$.
The coordinates $(r,\theta)$ are also referred to as geodesic coordinates.

\begin{exa}
Examples for spherically symmetric Riemannian manifolds, and for warped products of the form $\frak{I}\times_f N$, are the space forms.
For Euclidean spaces we have $\varphi(r)=r$; for spheres we have $\varphi(r)=\sin r$ and for hyperbolic spaces we have $\varphi(r)=\sinh r$.
\end{exa}
\smallskip

In this paper we examine the harmonicity of functions of the form
\begin{align*}
\psi:(0,s)\times_fN\rightarrow M\quad(r,\theta)\mapsto(\Phi(r),f_{\frak{e}}(\theta)),
\end{align*}
where $f_{\frak{e}}:N\rightarrow\Sph^{m-1}$ is an eigenmap with energy density $\frak{e}$, i.e. a harmonic function with constant energy density $\frak{e}$.
The energy of the map $\psi$ is given by
\begin{align*}
E(\psi)=c\int_0^{s}(\Phi'(r)^2+2\mathfrak{e}\tfrac{\varphi^2(\Phi(r))}{f(r)^2})f(r)^{n-1}dr,
\end{align*}
where $c$ is some positive constant.
The associated Euler Lagrange equation reads
\begin{align}
\label{euler}
\Phi^{''}(r)+(n-1)\tfrac{f'(r)}{f(r)}\Phi'(r)-2\mathfrak{e}\tfrac{\varphi(\Phi(r))\varphi'(\Phi(r))}{f(r)^2}=0.
\end{align}
Here we have $r\in(0,s)$ and $\Phi(r)$ has to satisfy some boundary conditions at $r=0$ and $r=s$,
which of course depend on the choice of $f$ and $\varphi$.

\smallskip

Below we consider the special case where the domain manifold is either the Euclidean space $M_{\kappa=0}=\R^{n+1},$ or the sphere $M_{\kappa=1}=\Sph^{n+1}$. 
These are precisely the cases considered in later sections.
In both cases it is useful to introduce the new variable 
$$t=\log(f(r)/f'(r)).$$
A straightforward calculation yields
\begin{equation*}
\tfrac{dt}{dr}=
   \begin{cases}
     e^{-t} & \text{for}\quad \kappa=0,\\
    2\cosh t & \text{for} \quad \kappa=1.
   \end{cases}
\end{equation*}
For $\kappa=0$ and in terms of the variable $t$, equation (\ref{euler}) is given by
\begin{align*}
\Phi^{''}(t)+(n-2)\Phi'(t)-2\mathfrak{e}\varphi(\Phi(t))\varphi'(\Phi(t))=0
\end{align*}
For $\kappa=-1$ and in terms of the variable $t$, equation (\ref{euler}) becomes
\begin{multline*}
\Phi^{''}(t)-\tfrac{1}{2}((n-3)\tanh t-(n-1))\Phi'(t)\\-\mathfrak{e}(1-\tanh t)\varphi(\Phi(t))\varphi'(\Phi(t))=0.
\end{multline*}

\subsubsection{Known constructions}
\label{known}
In this subsection we provide a list of known constructions of harmonic maps from warped products into spherically symmetric manifolds.

\smallskip

\noindent\textbf{Harmonic maps from balls to spheres.}
As already mentioned in the introduction, in \cite{jk} J\"ager and Kaul, constructed rotationally symmetric harmonic maps 
in $H^1_2(B^n,\Sph^n)$ solving the Dirichlet problem $\mbox{Dir}(n,\rho)$.

\smallskip

For reasons of completeness, we first give a definition of what it means for a map in the set
\begin{align*}
H^1_2(B^n,\Sph^n)=\{u\in H^1_2(B^n,\R^{n+1})\,\lvert\,u(x)\in\Sph^n\,\mbox{for a.e.}\, x\in B^n\}.
\end{align*}
to be harmonic.
The later needed definition of harmonic maps in $$H^1_2(B^n,\Sph^m)=\{u\in H^1_2(B^n,\R^{m+1})\,\lvert\,u(x)\in\Sph^m\,\mbox{for a.e.}\, x\in B^n\}$$ is completely analogous, and therefore omitted here. We closely follow the presentation in \cite{jk}, where we also borrow the notation.\\
For $u,v\in H^1_2(B^n,\Sph^n)$ we set
\begin{align*}
\langle du,dv\rangle=\sum_{i=1}^n\langle \partial_iu,\partial_i v\rangle,
\end{align*}
where we denote by $\partial_i$ the weak derivative with respect to the Cartesian coordinate $x_i$.
By $\langle\,\cdot\,,\,\cdot\,\rangle$ we denote the Euclidean scalar product.
The energy of $u\in\mathcal{S}^n=H^1_2(B^n,\Sph^n)$ is then set to be $E(u)=\tfrac{1}{2}\int_{B^n}\lvert du\lvert^2\omega_n$, and
critical points of the energy are called harmonic maps.
Here the term \lq critical\rq\, needs some explanation: for $u\in\mathcal{S}_n$ we introduce
\begin{align*}
\delta_u\mathcal{S}^n=\{v\in H^1_2(B^n,\R^{n+1})\,\lvert\,\langle u(x),v(x)\rangle=0\quad\mbox{for a.e.}\quad x\in B^n\},
\end{align*}
i.e. the space of vector fields along $u$.
Furthermore, we set
\begin{align*}
\delta_u\mathcal{S}^n_0=\delta_u\mathcal{S}^n\cap \mathring{H}^1_2.
\end{align*}
Then the first variation of the energy is given by
$$\delta_uE(v)=\int_{B^n}\langle du,dv\rangle \omega_n,$$
and $u$ is defined to be a critical point of $E$ if and only if $\delta_uE(v)=0$ for all $v\in\delta_u\mathcal{S}^n_0.$

\smallskip

We this preparation at hand we may now state some results of \cite{jk}.
Recall from the introduction that J\"ager and Kaul searched for maps of the form
\begin{align}
\label{map_u1}
u:B^n\rightarrow\Sph^n,\quad u(x)=\left(\begin{smallmatrix}
\tfrac{\sin(\Phi(r))}{r}x\\
\cos(\Phi(r))
\end{smallmatrix} \right),
\end{align}
where $\Phi:[0,1]\rightarrow[0,\pi]$ and $r=\lvert x\lvert$, solving the Dirichlet problem $\mbox{Dir}(n,\rho)$.
We now write the maps (\ref{map_u1}) in different coordinates in order to see that these maps fit into the setting of the previous subsubsection.
Let $(r,\theta)$ be polar coordinates on $B^n$ and endow $B^n$ with the metric $$dr^2+r^2d\theta^2_{\lvert\Sph^{n-1}}.$$
Furthermore, denote by $(\Phi,\Theta)$ geodesic coordinates on $\Sph^m$, such that the metric on $\Sph^m$ is given by
$$d\Phi^2+\sin(\Phi)^2d\Theta^2_{\lvert\Sph^{m-1}}.$$
In terms of these coordinates, the maps (\ref{map_u1}) are of the form
\begin{align*}
(r,\theta)\mapsto (\Phi(r),\theta).
\end{align*}
Clearly, these maps fit into the scheme considered into the previous subsubsection.
The Euler Lagrange equations associated to the energy functional are thus given by
\begin{align*}
\Phi^{''}(r)+(n-1)\tfrac{\Phi'(r)}{r}-\mathfrak{e}_1\tfrac{\sin(2\Phi(r))}{r^2}=0,
\end{align*}
where $\frak{e}_1=(n-1)/2$.
In \cite{jk} solutions of the preceding ODE with $\Phi(0)=0$ are provided.

\begin{theorem}
Let $3\leq n\leq 6$.
Then there exist infinitely many smooth harmonic maps from $B^n$ to $\Sph^n$ solving the Dirichlet problem $\mbox{Dir}(n,\tfrac{\pi}{2})$.
\end{theorem}

\smallskip

\noindent\textbf{Harmonic maps between spheres.}

\noindent We present two constructions of smooth harmonic maps between spheres of possibly different dimensions.
Note that there exist more constructions in the literature, which we do not discuss for reasons of shortness.
The interested reader might consult \cite{er} and the references therein.

\medskip

\noindent\textit{A: Construction by Corlette, Wand and Bizon, Chmaj}\\ 
\noindent 
Below assume that all appearing spheres are endowed with polar coordinates -- compare Subsection 1.2.1.
In \cite{wc}, Corlette and Wald examined the maps 
\begin{align*}
\psi_k:\Sph^n\rightarrow\Sph^m,\quad(r,\theta)\mapsto (\rho(r),f_{\frak{e}_k}(\theta)),
\end{align*}
for harmonicity. Here $f_{\frak{e}_k}:\Sph^{n-1}\rightarrow\Sph^{m-1}$ is an eigenmap with energy density $\frak{e}_k=k(k+n-2)/2$, $k\in\N$.
The Euler Lagrange equations associated to the energy functional of $\psi_k$ are given by
\begin{align*}
\ddot\rho(r)+(n-1)\cot r\dot\rho(r)=\frak{e}_k\tfrac{\sin(2\rho(r))}{\sin^2r}.
\end{align*}
Note that this setting clearly fits into the scheme considered in the previous subsection.
Furthermore, note that the case special case $m=n$ and $k=1$ had been considered earlier by Bizon and Chmaj  in \cite{BC}.

\smallskip

Using Morse theoretic methods, Corlette and Wald produced infinitely many smooth harmonic maps between spheres.

\begin{theorem}[Theorem 3.8 in \cite{wc}]
Suppose that $F:\Sph^n\rightarrow\Sph^m$ is an eigenmap with eigenvalue $\omega$, $m>1$, and
$(n-1)^2/4<\omega$.
Then are infinite sequences of critical points for the energy functional $E$, i.e. infinite sequences of smooth harmonic maps  from $\Sph^n$ to $\Sph^m$.
\end{theorem}
 
From this result one can immediately deduce the following theorem.

\begin{theorem}
For each $\ell\geq 3$ there exist infinitely many smooth harmonic self-maps of $\Sph^{\ell}$.
\end{theorem}

Note that this result has already been known to hold in dimensions $\ell\in\{3,4,5\}$ by the work of  Bizon and Chmaj \cite{BC}.
\smallskip

Finally we want to mention that one can use the same Ansatz to produce biharmonic maps between spheres, are more general between warped product spaces.
Such a construction has been provided by Montaldo, Oniciuc and Ratto - see \cite{mor0,mor} and the references therein.

\medskip

\noindent\textit{B: Construction by Smith}\\ 
Smith \cite{smith} invented the so-called harmonic Hopf- and Join- constructions to produce homotopically non trivial maps between spheres.
For the Hopf construction maps between doubly warped products and warped products are considered.
For the Join construction one considers maps from doubly warped products to doubly warped products.
Both methods reduce the construction of harmonic maps between spheres to solving a second order singular boundary value problem.
Below we recall a few details.

\bigskip

\noindent\textit{Hopf construction.}
For a continuous map $f:\Sph^{p_1}\times\Sph^{p_2}\rightarrow\Sph^{q-1}$,
the classical Hopf construction associated to $f$, is given by
$$H_f:\Sph^{p_1+p_2+1}\rightarrow\Sph^{q},\quad (x_1\sin t,x_2\cos t)\mapsto (f(x_1,x_2)\sin2t,\cos2t),$$
where $x\in\Sph^{p_1+p_2+1}$ is written uniquely (except for a set of measure zero) as $x=(x_1\sin t,x_2\cos t)$ for $x_1\in\Sph^{p_1}$, $x_2\in\Sph^{p_2}$
and $t\in[0,\tfrac{\pi}{2}]$. 
Smith \cite{smith} introduced the maps 
$$H(x_1\sin t,x_2\cos t)=(f(x_1,x_2)\sin r(t),\cos r(t)),$$
where $r:[0,\tfrac{\pi}{2}]\rightarrow[0,\pi]$,
which are homotopic to $H_f$ and determined conditions for them to be harmonic.
Namely, he showed that if $f$ is a bi-eigenmap with eigenvalues $\lambda_1, \lambda_2\in\N$,
and $r$ satisfies
\begin{align}
\label{ode_hopf}
\ddot r(t)+(p_1\cot t-p_2\tan t)\dot r(t)-\tfrac{1}{2}\left(\tfrac{\lambda_1}{\sin^2t}+\tfrac{\lambda_2}{\cos^2t}\right)\sin2r(t)=0,
\end{align}
with $r(0)=0$ and $r(\frac{\pi}{2})=\pi$, then $H:\Sph^{p_1+p_2+1}\rightarrow\Sph^{q}$ is harmonic. 
Recall, that $f:\Sph^{p_1}\times\Sph^{p_2}\rightarrow\Sph^{q-1}$ is called a bi-eigenmap, if the map
$f(\,\cdot\,,x):\Sph^{p_1}\rightarrow\Sph^{q-1}$ is an eigenmap for every $x\in\Sph^{p_2}$
and if the map
$f(x,\,\cdot\,):\Sph^{p_2}\rightarrow\Sph^{q-1}$ is an eigenmap for every $x\in\Sph^{p_1}$.

\smallskip

Note that this construction fits in the framework of this section. This becomes much clearer when one switches from Euclidean coordinates to polar coordinates.
Indeed,
let $\Sph^{p_1+p_2+1}$ and $\Sph^q$ be endowed with the doubly warped metric 
$$dt^2+\sin^2tds^2_{p_1}+\cos^2tds^2_{p_2},$$
and with the warped metric
$$dt^2+\sin^2tds^2_{q-1},$$
respectively.
Consider the maps $\rho:\Sph^{p_1+p_2+1}\rightarrow\Sph^q$ given by
\begin{align*}
(t,\mu,\nu)\mapsto(r(t),f(\mu,\nu)),
\end{align*}
where $f:\Sph^{p_1+p_2}\rightarrow\Sph^{q-1}$ is a bi-eigenmap with eigenvalues $\lambda_1, \lambda_2\in\N$.
The energy of $\rho$ is given by
\begin{align*}
E(\rho)=c\int_0^{\pi/2}(r'(t)^2+\lambda_{1}\tfrac{\sin^2r}{\sin^2t}+\lambda_2\tfrac{\sin^2r}{\cos^2t})\sin^{p_1}t\cos^{p_2}tdt,
\end{align*}
where $c$ is a positive constant. The associated Euler Lagrange equation is given by (\ref{ode_hopf}).
Smoothness across the focal varieties can be proven as in \cite{smith} or \cite{gastel}.

\bigskip

\noindent\textit{Join construction.}
Let $f_i:\Sph^{p_i}\rightarrow\Sph^{q_i}$, $i\in\lbrace 1,2\rbrace$ be two homogeneous polynomials.
In algebraic topology, the join of $f_1$ and $f_2$ is given by 
$$J_{f_{1},f_2}:\Sph^{p_1+p_2+1}\rightarrow\Sph^{q_1+q_2+1}, (x_1\sin t,x_2\cos t)\mapsto (f_1(x_1)\sin t,f_2(x_2)\cos t),$$
 where $x_1$ and $x_2$ are defined as above.
Smith \cite{smith} introduced maps $J$ homotopic to $J_{f_1,f_2}$, and determined conditions for them to be harmonic.
 More precisely, he considered
 eigenmaps $f_1$ and  $f_2$ with eigenvalues $\lambda_1$ and $\lambda_2$, respectively,  and made the ansatz
 $$J(x_1\sin t,x_2\cos t)=(f_1(x_1)\sin r(t),f_2(x_2)\cos r(t)),$$
where $r:[0,\tfrac{\pi}{2}]\rightarrow[0,\tfrac{\pi}{2}]$.
Smith proved that $J$ yields a harmonic map if $r$ satisfies
\begin{align}
\label{ode_join}
\ddot r(t)+(p_1\cot t-p_2\tan t)\dot r(t)-\tfrac{1}{2}\left(\tfrac{\lambda_1}{\sin^2t}-\tfrac{\lambda_2}{\cos^2t}\right)\sin2r(t)=0,
\end{align}
with $r(0)=0$, $r(\tfrac{\pi}{2})=\tfrac{\pi}{2}$ and $0\leq r\leq\tfrac{\pi}{2}$.

\smallskip

As above we now switch to polar coordinates to see that this construction fits into the framework of this section.
Let the spheres $\Sph^{p_1+p_2+1}$ and $\Sph^{q_1+q_2+1}$ be endowed with the doubly warped metric 
\begin{align*}
dt^2+\sin^2tds^2_{p_1}+\cos^2tds^2_{p_2}\quad\mbox{and}\quad dt^2+\sin^2tds^2_{q_1}+\cos^2tds^2_{q_2},
\end{align*}
respectively. 
Furthermore, let $f:\Sph^{p_1}\rightarrow\Sph^{q_1}$ and $g:\Sph^{p_2}\rightarrow\Sph^{q_2}$ be eigenmaps with energy density $\lambda_1$ and $\lambda_2$, respectively.
Consider the maps $\phi:\Sph^{p_1+p_2+1}\rightarrow\Sph^{q_1+q_2+1}$ given by
\begin{align*}
(t,\mu,\nu)\mapsto(r(t),f(\mu),g(\nu)).
\end{align*}
The energy of $\phi$ is given by
\begin{align*}
E(\phi)=c\int_0^{\pi/2}(r'(t)^2+\lambda_1\tfrac{\sin^2r}{\sin^2t}+\lambda_2\tfrac{\cos^2r}{\cos^2t})\sin^{p_1}t\cos^{p_2}tdt,
\end{align*}
where $c$ is a positive constant.
The associated Euler Lagrange equation is given by (\ref{ode_join}).
Again, smoothness across the focal varieties can be proven as in \cite{smith} or \cite{gastel}.

\smallskip

For mored details on the study of the Hopf and the Join construction we refer the reader to the papers \cite{ding,gastel,pr,ratto} and to the references therein.

\section{Harmonic maps from balls to spheres}
\label{sec1}
In this section we provide new harmonic maps from balls to spheres.
The first subsection contains the prove of Theorem A, the second subsection provides Theorem\,B.

\subsection{ Dirichlet problem and harmonic maps from balls to spheres}
The purpose of this section is to prove Theorem\,A.
With the exception of one special map, all maps constructed in this subsection are smooth --- compare Theorem\,\ref{sol}.

\bigskip

\noindent\textbf{Preliminaries.}\\
Throughout this section let $f_{\mathfrak{e}_k}:\Sph^{n-1}\rightarrow\Sph^{m-1}$ be an eigenmap
with energy density 
$$\mathfrak{e}_k=k(k+n-2)/2,$$ $k\in\N$.
We are interested in constructing solutions of the Dirichlet problem $\mbox{Dir}(n,m,\rho,e_k)$ (see the introduction) which are of the form 
$$u(x)=(f_{\mathfrak{e}_k}(\tfrac{x}{\lvert x\lvert})\sin \Phi(\lvert x\lvert),\cos\Phi(\lvert x\lvert)),$$ 
where $\Phi:B^n\rightarrow[0,\pi]$ is a $C^2$ map.
As in Subsection 1.2.2, we introduce polar coordinates $(r,\theta)$ on $B^n$ and geodesic coordinates $(\Phi,\Theta)$ on $\Sph^m$.
In terms of these coordinates, the maps $u$ are given by
\begin{align}
\label{maps}
\varphi:B^n\rightarrow\Sph^m,\quad (r,\theta)\mapsto (\Phi(r),f_{\frak{e}_k}(\theta)),
\end{align}
By Subsection\,1.2.1 the Euler Lagrange equation associated to the energy functional of $u$ is given by
\begin{align}
\label{el_1}
\Phi^{''}(r)+(n-1)\tfrac{\Phi'(r)}{r}-\mathfrak{e}_k\tfrac{\sin(2\Phi(r))}{r^2}=0.
\end{align}
Following \cite{jk} we define $\Psi:(-\infty,0]\rightarrow [0,\pi]$ to be $\Psi(t)=\Phi(e^t)$.
Thus equation (\ref{el_1}) becomes
\begin{align}
\label{ode_psi}
\Psi^{''}(t)+(n-2)\Psi'(t)-\mathfrak{e}_k\sin(2\Psi(t))=0.
\end{align}
In terms of $\Psi$ the energy of (\ref{maps}) is given by
\begin{align*}
E(\varphi)=c\int_{-\infty}^0(\Psi'(t)^2+2\mathfrak{e}_k\sin^2(\Psi(t)))e^{(n-2)t}dt,
\end{align*}
where $c\in\R_+$.

\smallskip

Writing the ordinary differential equation (\ref{ode_psi}) as a system of first order differential equations we get
\begin{align}
\label{system2}
\tfrac{d}{dt}\left(\begin{smallmatrix}
\Psi\\ 
\Psi'
\end{smallmatrix} \right)=\left(\begin{smallmatrix}
\Psi'\\ 
-(n-2)\Psi'+\frak{e}_k\sin(2\Psi)
\end{smallmatrix} \right).
\end{align}
Clearly, the critical points of this system are given by $(\Psi,\Psi')=(k\pi/2,0)$ for $k\in\Z$.

\bigskip

\noindent\textbf{Harmonic maps from balls to spheres with finite energy.}\\
Following Lemma 2.13 in \cite{jk}, we classify harmonic maps of the form (\ref{maps}) with finite energy.
As in \cite{jk}, we prove that these solutions are directly related to the trajectories of equation (\ref{ode_psi}) which connect the critical points $(0,0)$ or $(\pi,0)$ with $(\pi/2,0)$ in the phase plane $(\Psi,\Psi')$.

\smallskip

The next lemma is needed as preparatory lemma.

\begin{lemma}
\label{stripe}
For a non-constant solution $\Psi$ of (\ref{ode_psi}) with finite energy there exists a $k_0\in\Z$ such that $k_0\pi<\Psi <(k_0+1)\pi$.
\end{lemma}
\begin{proof}
Set 
\begin{align*}
\label{fct_v}
V=\Psi'^2-2\frak{e}_k\sin^2\Psi.
\end{align*}
By (\ref{ode_psi}) we have
\begin{align*}
V'=-2(n-2)\Psi'^2,
\end{align*}
i.e. $V$ is a Lyapunov function for (\ref{ode_psi}).
Since $\Psi'(t)=\Phi'(e^t)e^t$ and we are searching for solutions of (\ref{el_1}) with $\Phi'(0)<\infty$,
we have $$\lim_{t\rightarrow -\infty}\Psi'(t)=0.$$ Consequently, $V(-\infty)\leq 0$. Therefore there cannot exist $t_0\in(-\infty,\infty)$ such that $\Psi(t_0)$ is a multiple of $\pi$.
\end{proof}

As an immediate consequence of the previous lemma and its proof we obtain, that a solution of (\ref{ode_psi})
can not cover both poles of the sphere. 

\smallskip

We are now ready to classify harmonic maps of the form (\ref{maps}) with finite energy.

\begin{lemma}
\label{fin_energy}
Let $\Phi(r)=\Psi(\ln r)$ be as in (\ref{maps}) such that $\varphi$ has finite energy
and $0\leq\Psi\leq\pi$. Then\vspace{0.2cm}\\
if $n=2$
\begin{enumerate}
\renewcommand{\labelenumi}{(\roman{enumi})}
\item we have $\Phi(r)=2\arctan(cr^k)$ or $\rho(r)=\pi-2\arctan(cr^k)$ with some constant $c\geq 0$,
\end{enumerate}
\vspace{0.1cm}
if $n\geq 3$
\begin{enumerate}
\setcounter{enumi}{1}
\renewcommand{\labelenumi}{(\roman{enumi})}
\item $\Psi$ is either constant with values $0,\pi/2$ or $\pi$; or
\item $\Psi$ is extendable to a solution of (\ref{ode_psi}) on $\R$ and\\
$\lim_{t\rightarrow -\infty}(\Psi(t),\Psi'(t))=(0,0)$ or $(\pi,0)$\\
$\lim_{t\rightarrow \infty}(\Psi(t),\Psi'(t))=(\pi/2,0)$.
\end{enumerate}
\end{lemma}
\begin{proof}
Define the function $V$ as in the proof of Lemma\,\ref{stripe}, i.e. $V=\Psi'^2-2c_k\sin^2(\Psi)$.
Now the proof follows along the lines of the proof of Lemma 2.13 in \cite{jk}.
\end{proof}

\begin{remark}
Note that the assumption $0\leq\Psi\leq\pi$ in Lemma\,\ref{fin_energy} is not a real restriction.
By the proof of Lemma\,\ref{fin_energy} it follows that $\lim_{t\rightarrow -\infty}\Psi(t)=\ell\pi$ for a $\ell\in\Z$.
Since (\ref{ode_psi}) is $2\pi$-periodic in $\Psi$, we can restrict ourselves to the cases $\ell=0$ and $\ell=1$.
If $\ell=0$ and $-\pi\leq\Psi\leq 0$, we may consider $-\Psi$ instead of $\Psi$.
Similarly, if $\ell=1$ and $\pi\leq\Psi\leq 2\pi$, we may consider $2\pi-\Psi$ instead of $\Psi$.
\end{remark}

\bigskip

\noindent\textbf{Study of the plane system.}\\
Following \cite{jk}, we rewrite the system (\ref{system2}) in terms of the functions
$q(t)=2\Psi(t)-\pi$ and $p(t)=q'(t)$:
\begin{equation}
\label{sys2}
\left(\begin{smallmatrix}q'\\ p'
\end{smallmatrix} \right)=\left(\begin{smallmatrix}
p\\ 
-(n-2)p-2\frak{e}_k\sin(q)
\end{smallmatrix} \right)=:V(q,p).
\end{equation}
Clearly, the critical points of this system are given by $(q,p)=(k\pi,0)$ for $k\in\Z$.

\smallskip

Note that the critical points $(0,0)$ and $(\pi/2,0)$ of (\ref{system2}) in the $(\Psi,\Psi')$ phase space correspond to the critical points $(-\pi,0)$ and $(0,0)$ of (\ref{sys2}) in the $(q,p)$ phase space, respectively.
By the same reasoning as in \cite{jk}, we show that there exists exactly one invariant curve in the $(q,p)$-plane which connects  the critical points $(-\pi,0)$ and $(0,0)$. Furthermore, the trajectories on this curve are uniquely determined up to a translation in $t$.
We denote such a trajectory by $(Q,P):\R\rightarrow\R^2$.
Clearly, the constant $c$ in 
\begin{align}
\lim_{t\rightarrow -\infty}\tfrac{1}{2}e^{-t}P(t)=\Phi'(0)=:c
\end{align}
determines $(Q,P)$ uniquely. As in \cite{jk}, we define $(Q_n,P_n)$ by choosing $c=1$.
In the following lemma, which is an analogue of Lemma 2.18 in \cite{jk},
we show that $Q_n$ oscillates around $0$ for $t\rightarrow\infty$ and $3\leq n\leq k_0$, where $$k_0:=\lfloor 2(1+k+\sqrt{k})\rfloor.$$
Moreover, it is proven that $\Theta_n$ increases strictly, tending to $0$ as $t\rightarrow\infty$, if $n\geq k_0+1$. 

\begin{lemma}
Let $k$, as in (\ref{maps}), be given.
Introduce the functions $R_n$ and $\Theta_n$ by
\begin{align*}
Q_n(t)+iP_n(t)=R_n(t)e^{i\Theta_n(t)}.
\end{align*}
Then we have 
\begin{align*}
0<R_n(t)\leq ce^{\mu t}\quad\forall t\in\R,
\end{align*}
where $c=c(n,\mu)>0$ is a constant and $\mu>\mbox{Re}\,\lambda_{n,k}^+$. 
Then we get
\begin{enumerate}
\renewcommand{\labelenumi}{(\roman{enumi})}
\item for $3\leq n \leq k_0$ the invariant curve is a spiral with center $(0,0)$ satisfying
\begin{align*}
\lim_{t\rightarrow\infty}\tfrac{\Theta_n(t)}{t}=-\tfrac{1}{2}\sqrt{-(4+8k-4k^2 -4n-4kn+n^2)},
\end{align*}
\item for $n\geq k_0+1$ we have
\begin{align*}
\pi+\arctan\lambda_{n,k}^+<\Theta_n(t)<\pi\quad\mbox{and}\quad\lim_{t\rightarrow\infty}\Theta_n(t)=\pi+\arctan\lambda_{n,k}^+.
\end{align*}
\end{enumerate}
\end{lemma}

\bigskip

\noindent\textbf{Dirichlet problem.}\\
J\"ager and Kaul studied the Dirichlet problem $\mbox{Dir}(n,\rho),$ which depends on a parameter $\rho\in[0,\pi]$ (compare page 154 in \cite{jk}):
Find a rotationally symmetric harmonic map $u\in H^1_2(B^n,\Sph^n)$ satisfying the boundary condition
$$u_{\lvert\partial B^n}=b_{\rho}:x\mapsto (x\sin\rho,\cos\rho).$$
Here we will study the following generalized Dirichlet problem.
\begin{dir2}
Find a map $u\in H^1_2(B^n,\Sph^m)$ of the form (\ref{maps}) satisfying the boundary condition
$$u_{\lvert\partial B^n}=b_{\rho}:x\mapsto (f_{\frak{e}_k}(x)\sin\rho,\cos\rho).$$
\end{dir2} 
Obviously, for $k=1$ and $n=m$ the Dirichlet problem $\mbox{Dir}(n,n,\rho,f_{\frak{e}_1})$ coincides with $\mbox{Dir}(n,\rho)$.

\smallskip

In the next theorem we characterize the solutions of the Dirichlet problem  $\mbox{Dir}(n,m,\rho,f_{\frak{e}_k})$.
However, we need to introduce some notation first.
Following \cite{jk} we set 
\begin{align*}
\Psi_n=\tfrac{1}{2}(Q_n+\pi).
\end{align*}
Hence $\Psi_n$ satisfies the ODE (\ref{ode_psi}) with 
\begin{align*}
\lim_{t\rightarrow -\infty}\Psi(t)=0\quad\mbox{and}\quad\lim_{t\rightarrow -\infty}e^{-t}\Psi'(t)=1.
\end{align*}
With this notation at hand we can formulate the following theorem.

\begin{theorem}
\label{sol}
The solutions of $\mbox{Dir}(n,m,\rho,f_{\frak{e}_k})$ are characterized as follows.\vspace{0.3cm}\\
$n=2$.\\
\indent\indent\indent\indent $\Phi(r)=2\arctan(r^k\tan(\tfrac{\rho}{2}))$\\
or\\ 
\indent\indent\indent\indent $\Phi(r)=2\arctan(\tfrac{1}{r^k}\tan(\tfrac{\rho}{2}))$\vspace{0.2cm}\\
with $\arctan(\infty):=\pi/2$.\vspace{0.3cm}\\
$n\geq 3$.\\ The equator map $u_0$, given by $\Phi=\pi/2$, is the only discontinuous harmonic  map for $\rho=\pi/2$.
All other solutions are smooth and they either cover the northpole or the southpole (exclusively).\\
The northpole is covered if and only if there exists a number $\tau\in\R\cup{-\infty}$ satisfying
\begin{equation}
\Phi(r)=\Psi_n(\tau+\ln r),\quad\mbox{for all}\quad r\in[0,1],\quad\quad \Psi_n(\tau)=\rho,
\end{equation}
where $\Psi_n(-\infty):=0$.
\end{theorem}

\begin{center}
\begin{figure}[h]
\caption{Qualitative phase space diagrams}
\begin{picture}(1200,200)(40,-90)
\put(-10,0){\vector(1,0){250}}
\put(0,-60){\vector(0,1){150}}
%\put(-5,100){\line(1,0){230}}
\multiput(188,-3)(0,10){10}{\line(0,1){6}}
\multiput(188,-28)(0,-10){5}{\line(0,1){6}}
\multiput(205,15)(0,-10){6}{\line(0,1){6}}
\put(188,-5){\makebox(0,0)[tc]{$\frac{\pi}{2}$}}
\put(225,-7){\makebox(0,0)[tc]{$\Psi_n$}}
\put(-9,80){\makebox(0,0)[cr]{$\Psi'_n$}}
\put(120,-40){\makebox(0,0)[cr]{$3\leq n\leq k_0$}}
\thicklines
%\qbezier(0,0)(20,70)(100,60)
%\qbezier(100,60),(140,55),(180,25)
%\qbezier(180,25)(192,15)(195,0)
%\qbezier(195,0)(193,-14)(185,-20)
%\qbezier(185,-20)(184,-21)(178,-17)
\qbezier(0,0)(20,70)(100,60)
\qbezier(100,60),(140,55),(190,20)
\qbezier(190,20)(202,9)(205,0)
\qbezier(205,0)(206,-15)(195,-20)
\qbezier(195,-20)(187,-21)(182,-15)
\qbezier(182,-15)(179,-10)(181,0)
\qbezier(181,0)(184,3)(187,3)
\end{picture}

\begin{picture}(-90,-400)(0,-104)
\put(-10,0){\vector(1,0){220}}
\put(0,-60){\vector(0,1){150}}
\put(120,-40){\makebox(0,0)[cr]{$n>k_0$}}
%\put(-5,100){\line(1,0){230}}
\multiput(188,-3)(0,10){1}{\line(0,1){6}}
\multiput(188,-28)(0,-10){0}{\line(0,1){6}}
\put(188,-5){\makebox(0,0)[tc]{$\frac{\pi}{2}$}}
\put(215,-7){\makebox(0,0)[tc]{$\Psi_n$}}
\put(-9,80){\makebox(0,0)[cr]{$\Psi'_n$}}
\thicklines

\qbezier(0,0)(20,70)(83,60)
\qbezier(83,60),(123,55),(173,20)
\qbezier(173,20)(185,9)(188,0)
\end{picture}
\end{figure}
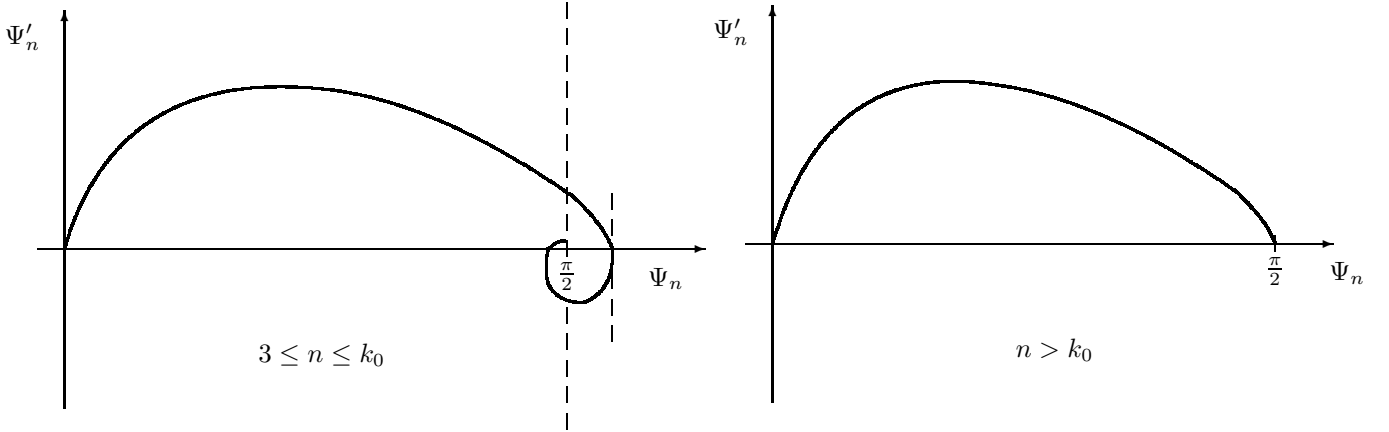
\end{center}

Below we discuss some applications of the previous results.
Figure\,1 contains a sketch of the phase space diagrams for the cases $3\leq n\leq k_0$ and $n>k_0$, respectively.
For the cases $3\leq n\leq k_0$ the curve spirals around the focal point $(\Psi_n,\Psi'_n)=(\pi/2,0)$.
Following \cite{jk}, we denote the maximal value of $\Psi_n$ by $\rho_n$ and the smallest local minimum by $\sigma_n$.
As in \cite{jk} it follows by comparison arguments that
\begin{align*}
\pi/2<\rho_{k_0}<\cdots <\rho_3<\pi.
\end{align*}
For the cases $n>k_0$, the map $\Psi_n$ increases monotonically and goes to $\pi/2$ as $t$ goes to $\infty$.
By Theorem\,\ref{sol} each intersection point of a vertical line with the curve corresponds to a solution of the Dirichlet problem
$\mbox{Dir}(n,m,\rho,f_{\frak{e}_k})$.

\begin{cor}
\label{sol}
The number of solutions of $\mbox{Dir}(n,m,\rho,f_{\frak{e}_k})$ is\vspace{0.3cm}\\
in the case $n=2$: one if $\rho\in[0,\pi)$, zero for $\rho=\pi$,\vspace{0.2cm}\\
in the case $3\leq n\leq k_0$: one for $\rho\in[0,\sigma_n)$, two if $\rho=\sigma_n$, an odd number if $\rho\in(\sigma_n,\pi/2)$,
countable infinite if $\rho=\pi/2$, an even number for $\rho\in(\pi/2,\rho_n)$, one for $\rho=\rho_n$ and zero if $\rho>\rho_n$,\vspace{0.2cm}\\
in the case $n\geq k_0+1$: one if $\rho\in[0,\pi/2)$ and zero for $\rho\geq\pi/2$.
\end{cor}

\begin{remark}
J\"ager and Kaul \cite{jk} describe the relation between the energy and the parameter $\rho$.
For the present situation this can be done analogously, with case distinction $n=2$, $3\leq n\leq k_0$ and $n\geq k_0+1$.
\end{remark}

\begin{cor}
\label{ball_sphere}
Let eigenmaps $f_{\frak{e}_k}:\Sph^{n-1}\rightarrow\Sph^{m-1}$ with energy density $\frak{e}_k=k(k+n-2)/2$, $k\in\N$, be given. Furthermore, assume $n\geq 3$ and $m\geq 2$.
Then there exist infinitely many harmonic maps $u:B^{n}\rightarrow\Sph^{m}$. 
\end{cor}
\begin{proof}
For given $m\geq 2$, choose $k$ large enough such that $k_0\geq m$.
The claim is thus established by  Corollary\,\ref{sol}.
\end{proof}

In summary, Theorem\,\ref{sol} and its corollaries establish Theorem\, A of the introduction. 

\begin{thma}
%\label{dirichlet}
Let $k_0:=\lfloor 2(1+k+\sqrt{k})\rfloor$.
For each $3\leq n\leq k_0$ there exists an explicit $\rho_n\in[\pi/2,\pi]$, such that $\mbox{Dir}(n,m,\rho,f_{\frak{e}_k})$ has at least one solution for $\rho\leq\rho_n$ and no solution if $\rho>\rho_n$.
Furthermore, for $\rho=\pi/2$, there exist infinitely many 
solutions to the Dirichlet problem $\mbox{Dir}(n,m,\rho,f_{\frak{e}_k})$.
However, for $n>k_0$, the Dirichlet problem $\mbox{Dir}(n,m,\rho,f_{\frak{e}_k})$ does not admit any solution for $\rho\geq\pi/2$.
\end{thma}

\bigskip

\noindent\textbf{Stability of the equator map}\\
Finally we state stability properties of the singular solution $u_0$.
Recall that a harmonic map is said to be stable if its second variation is non-negative. 

\begin{theorem}
The map $u_0$ is
\begin{enumerate}
\renewcommand{\labelenumi}{(\roman{enumi})}
\item  unstable if $3\leq n\leq k_0$,
\item absolute minimum of the energy on the class
$$\mathfrak{C}=\{u\in H^1_2(B^n,\Sph^m)\,\lvert\,u=u_0\quad\mbox{on}\quad \partial B^n\}.$$
\end{enumerate}
\end{theorem}

The proof is omitted since
one can easily adapt the proof of Theorem\,2, to the present situation.

\subsection{Twisted harmonic maps from $B^{n}$ into $\Sph^{2m}$}
The aim of this section is to construct further harmonic maps from balls to spheres, namely to prove Theorem\,B.
In order to do so we relax the conditions on the maps $u$, i.e.
we consider maps from balls to spheres of even dimension which are 
of a more general form than the maps considered in \cref{sec1}.
It will become clear below why our construction just works in the case where the target manifold is an even dimensional sphere.
The harmonic maps constructed in this section are contained in $H^{1}_{2}(B^{n},\R^{2m+1})$, i.e. they are not necessarily smooth.

\smallskip

In what follows we denote by $D_t$ the rotational $2\times 2$ matrix
\begin{align*}
D_t=\left(\begin{smallmatrix}
\cos t&-\sin t\\
\sin t&\cos t
\end{smallmatrix} \right).
\end{align*}
and let $R_t\in\mbox{Mat}(2m,\R)$ be a rotational block matrix with blocks of size $2\times 2$, whose entries are $D_t$ on the diagonal and $0_2$ everywhere else.
Below we examine the harmonicity of the maps $u:B^{n}\rightarrow\Sph^{2m}$ given by
\begin{align*}
u(x)=\left(\begin{smallmatrix}
R_{g(r)}f_{\frak{e}_k}(x/r)\sin(\Phi(r))\\
\cos(\Phi(r))
\end{smallmatrix} \right),
\end{align*}
where $r=\lvert x\lvert$ and $f_{\frak{e}_k}:\Sph^{n-1}\rightarrow\Sph^{2m-1}$ is a eigenmap with energy density $\frak{e}_k=k(k+n-2)/2$.
For the special case $g\equiv 0$, these maps obviously reduce to maps considered in \cref{sec1}.

\smallskip

The energy of the map $u$ is given by
\begin{align*}
E(u)=\tfrac{\omega_{n}}{2}\int_0^1\left(\Phi'(r)^2+(\tfrac{2\frak{e}_k}{r^2}+g'(r)^2)\sin^2\Phi(r)\right)r^{n-1}dr.
\end{align*}
The associated Euler Lagrange equation reads
\begin{align}
\label{el-2}
\Phi^{''}(r)+(n-1)\tfrac{\Phi'(r)}{r}-\frak{e}_k\tfrac{\sin(2\Phi(r))}{r^2}-\tfrac{\sin(2\Phi(r))}{2}g'(r)^2=0.
\end{align}
In terms of $t=\log r$ this equation reads
\begin{align}
\label{el-3}
\Phi^{''}(t)+(n-2)\Phi'(t)-\frak{e}_k\sin(2\Phi(t))-\sin(2\Phi(t))g'(t)^2=0.
\end{align}
In comparison with equation (\ref{ode_psi}), the additional term $\sin(2\Phi(t))g'(t)^2$ comes in.
Note however, that we can choose $g\in C^{\infty}(\R)$ arbitrarily!
If $g$ is constant, this equation reduces to (\ref{ode_psi}).
If $g$ is linear, (\ref{el-3}) is autonomous and the methods from \cref{sec1} apply.
For all remaining possible choices of $g$, the ODE (\ref{el-2}) is not autonomous and thus the machinery from \cite{jk} cannot be used. 

\smallskip

In this manuscript we restrict ourself to the case where $g$ is a linear function (in $t$).
Note that $g(t)=ct$ transforms into $$g(r)=c\ln(r).$$
However, for this choice of $g$, $u(0)$ is not determined. 
Hence we introduce the function $\hat{u}$, which we set to be equal to $u$ on $B^{n}\setminus\{0\}$ and equal to some (arbitrary) $u_0\in\Sph^{m+1}$ at $x=0$.
First we show that $\hat{u}$ is contained in $H^{1}_{2}(B^{n},\R^{2m+1})$.

\begin{lemma}
If $g(r)=c\ln(r)$, we have $\hat{u}\in H^{1}_{2}(B^{n},\R^{2m+1})$ for $n\geq 3$.
\end{lemma}
\begin{proof}
Below let $\Omega=B^n_{\epsilon}\subset\R^m$ be the open ball centered at the origin with radius $\epsilon$.
The ball of radius $\epsilon=1$ will still be denoted by $B^n$. 
For the sake of readability we will write $g(r)$ instead of $c\ln(r)$ throughout the proof.
We indicate the step in which we use the special form of $g$.

\smallskip

We show that each component $\hat{u}_i$ of $\hat{u}$ is contained in $H^{1},_{2}(B^{n})$.
Since $H^{1}_{2}(B^{n})$ is a vector space, it is sufficient to show that for each $j\in\{1,\cdots, 2n\}$ the functions 
\begin{align*}
h(x)=\cos(g(r))f_{\frak{e}_k}(x/r)_j\quad\mbox{and}\quad i(x)=\sin(g(r))f_{\frak{e}_k}(x/r)_j 
\end{align*} 
are in $H^{1}_{2}(B^{n})$.
Here $f_{\frak{e}_k}(x/r)_j$ denotes the $j$-th component of $f_{\frak{e}_k}(x/r)$.
In what follows we restrict ourselves to proving the desired result for $h$; the considerations for $i$ are analogous.

\smallskip

Note that $h\in L^2(B^n)$.
Indeed, we have
\begin{align*}
\int_{B^n}h(x)^2dx=\int_{B^n}\cos(g(r))^2{f_{\frak{e}_k}(x/r)_j}^2dx\leq \int_{B^n}dx<\infty.
\end{align*}
Outside the origin, the partial derivatives of $h$ are computed as
\begin{align}
\label{weak}
\tfrac{\partial}{\partial x_i}(\cos(g(r)){f_{\frak{e}_k}(x/r)_j})=&-\tfrac{\sin(g(r))}{r^2}x_i{f_{\frak{e}_k}(x/r)_j}\\\notag&+\sum_{\ell=1}^n\tfrac{\partial {f_{\frak{e}_k}(y)_j}}{\partial y_{\ell}}(\tfrac{1}{r}\delta_{\ell,i}-\tfrac{x_ix_{\ell}}{r^3})\cos(g(r)),
\end{align}
where $y=x/r$.
Note that in the preceding formula we made use of the special form of $g$.
Next we prove that (\ref{weak}) defines the weak derivatives of $h$ on the entire domain $B^n$.

\smallskip

Let $\eta\in C^{\infty}_c(B^n)$.
For any $\epsilon>0$ the Theorem of Gauss yields 
 \begin{align}
 \label{gauss}
 \int_{\epsilon<\lvert x\lvert <1}\tfrac{\partial}{\partial x_i}(h\eta)dx=\int_{\lvert x\lvert=\epsilon}h(x)\eta(x)\nu_i(x)dS,
 \end{align}
 where $dS$ is the $(n-1)$-dimensional measure on the surface of $B(0,\epsilon)$, and $\nu_i$ is such that $\nu=(\nu_1,\cdots,\nu_n)$
 is the unit normal pointing toward the interior of this ball.
 Furthermore, since $\eta$ is bounded and continuous, we have
 \begin{align*}
\int_{\lvert x\lvert=\epsilon}\lvert h(x)\eta(x)\nu_i(x)\lvert dS\leq c\lvert\lvert\eta\lvert\lvert_{\infty}\omega_{n}\epsilon^{n-1}.
 \end{align*} 
 The right hand side converges to $0$ as $\epsilon$ converges to $0$.
 Hence, from (\ref{gauss}) we have
 \begin{align*}
  \int_{B^n}\cos(g(r)){f_{\frak{e}_k}(x/r)_j}\tfrac{\partial \eta }{\partial x_i}dx= -\int_{B^n}\tfrac{\partial h}{\partial x_i}\eta dx,
 \end{align*}
 i.e. (\ref{weak}) defines the weak derivatives of $h$ on the entire domain $B^n$.
 
 \smallskip
 
It remains to show that the derivatives (\ref{weak}) are contained in $L^2(B^n)$.
We show that $\tfrac{\sin(g(r))}{r^2}x_i{f_{\frak{e}_k}(x/r)_j}$ is contained in this Lebesgue space, the remaining  summands are treated similarly.
 Clearly,
  \begin{align*}
  \int_{B^n}\tfrac{\sin^2(g(r))}{r^4}x_i^2{f_{\frak{e}_k}(x/r)_j}^2dx\leq  \int_{B^n}\tfrac{1}{r^2}dx=\omega_{n}\int_{0}^1r^{n-3}dr<\infty
 \end{align*}
 if and only if $n-3>-1$. Since $n\in\N$, the later condition is equivalent to $n\geq 3$.
\end{proof}

We can now show the existence of infinitely many harmonic maps $u:B^{n}\rightarrow\Sph^{2m}$ in $H^{1}_{2}(B^{n},\R^{2m+1})$ for every $n\geq 3$, i.e. Theorem\, B of the introduction.

\begin{thmb}
Let eigenmaps $f_{\frak{e}_k}:\Sph^{n-1}\rightarrow\Sph^{2m-1}$ with energy density $\frak{e}_k=k(k+n-2)/2$, $k\in\N$, be given.
 Then there exist infinitely many harmonic maps $u:B^{n}\rightarrow\Sph^{2m}$  in $H^{1}_{2}(B^{n},\R^{2m+1})$ for every $n\geq 3$.
\end{thmb}
\begin{proof}
Since the proof works as the proof of Theorem\,\ref{ball_sphere}, we omit some details.
Plugging $g(t)=ct$ in (\ref{el-3}) this equation becomes
\begin{align*}
\Psi^{''}(t)+(2n-2)\Psi'(t)-((2n-1)+c^2)\tfrac{\sin(2\Psi(t))}{2}=0.
\end{align*}
Following \cite{jk} we introduce
$q(t)=2\Psi(t)-\pi$ and $p(t)=q'(t)$ and thus obtain the following system
\begin{align*}
\left(\begin{smallmatrix}
q'(t)\\ 
p'(t)
\end{smallmatrix} \right)=\left(\begin{smallmatrix}
p(t)\\ 
-(2n-2)p(t)-((2n-1)+c^2)\sin(q)
\end{smallmatrix} \right)=:V(q,p).
\end{align*}
The singular points of the vector field $V$ are given by $p_1=(0,0)$ and $p_2=(-\pi,0)$.
A straightforward calculation yields
\begin{align*}
DV_{p_i}=\left(\begin{smallmatrix}
0&1\\ 
\pm((2n-1)+c^2)&-(2n-2)
\end{smallmatrix} \right),
\end{align*}
where the plus sign occurs for $i=2$ and the minus sign for $i=1$.
For $i=1$ the eigenvalues of this matrix are given by
\begin{align*}
\lambda_{1,2}=-(n-1)\pm\sqrt{(n-2)^2-2-c^2}.
\end{align*}
Clearly, for given $n$, one can choose $c$ such that these eigenvalues have negative real part and non-vanishing imaginary part.
Thus $p_1$ is a focal point, which again implies the existence of infinitely many harmonic maps $u:B^{2n}\rightarrow\Sph^{2n}$.
\\
For $i=2$ the eigenvalues of this matrix are given by
\begin{align*}
\lambda_{1,2}=1-n\pm\sqrt{n^2+c^2}.
\end{align*}
Both of them are real numbers, one is positive, the other negative. Thus $p_2$ is a saddle point.
By going along the lines of the proof of \cite{jk} we finish the proof.
\end{proof}

It is now possible to construct twisted harmonic maps solving a Dirichlet problem.
Due to the analogy with the previous section we forgo this here.

\bibliography{mybibfile}

\end{document}